\documentclass[12pt]{amsart}
\usepackage[dvips]{graphicx}
\usepackage[latin1]{inputenc}
\usepackage{amsmath,amsthm,amsopn,amstext,amscd,amsfonts,amssymb}
\usepackage{amssymb}
\usepackage{amsfonts}
\usepackage{lscape}
\long
\def\salta#1{\relax}
\newcommand{\R}{{I\!\!R}}
\newcommand{\N}{{I\!\!N}}

\newcommand{\car}{{\raise2pt\hbox{$\chi$}}}

\renewcommand{\t }{\tau }

\newtheorem{Theorem}{Theorem}[section]
\newtheorem{Corollary}[Theorem]{Corollary}

\newtheorem{Lemma}[Theorem]{Lemma}

\newtheorem{remark}[Theorem]{Remark}

\newcommand{\C}{\mathcal{C}}

 \newcommand{\NN}{\mathcal{N}}

 \newcommand{\W}{\mathcal{W}}

\newcommand{\I}{\mathcal{I}}

\begin{document}
\title[Multiple solutions of a Hardy-Sobolev  Equation.]{Multiple solutions of an elliptic Hardy-Sobolev equation with critical exponents on compact Riemannian manifolds.}
\author[Y. Maliki, F.Z. Terki]{Y. Maliki$^*$ and F.Z. Terki}
\address{Y. Maliki, F.Z. Terki \hfill \break\indent D\'epartement de
Math\'ematiques, Universit\'e Abou Bakr Belkaïd, Tlemcen, \hfill\break%
\indent Tlemcen 13000, Algeria.} \email{\texttt{malyouc@yahoo.fr,
fatimazohra113@yahoo.fr}}

\date{}
\maketitle

\begin{abstract} On a compact Riemannian manifold, we prove the existence of multiple solutions for an elliptic equation with critical Sobolev growth and  critical Hardy potential.
\end{abstract}

\section{Introduction}

Let $(M,g$) be a Riemannian manifold of dimension $n\ge3$ . For a fixed point $ p$ in $M$, we define the function $\rho_p$ on $M$ as follows
\begin{equation*}
    \rho_p(x)=\left\{
       \begin{array}{ll}
         dist_g(p,x), &x\in B(p,\delta_g),\\
         \delta_g, &x\in M\setminus B(p,\delta_g)
       \end{array}
     \right.
\end{equation*}
where, $\delta_g$ denotes the injectivity radius of $M$. Let $h$ be a continuous functions on $M$. Consider on $M\setminus\{p\}$ the following Hardy-Sobolev equation:
\begin{equation} \label{eq1.1}
\Delta _{g}u-\frac{h}{\rho_p ^{2}(x)}u=|u|^{2^*-2}u,
\tag{$E_h$}
\end{equation}
where $2^*=\frac{2n}{n-2}$ is the Sobolev critical exponent.\\
In this paper, we are interested in the study of existence of multiple solutions of equation \eqref{eq1.1}. When dropping the singular term $ \frac{1}{\rho_p ^{2}(x)}$ from equation \eqref{eq1.1} we fall in the  so called  Yamabe  equation which is very known in the literature and whose origin comes from the study of conformal deformation of the metric to constant scalar curvature. A positive solution $u$ of the Yamabe equation provides a conformal metric $g'=u^{\frac{4}{n-2}}g$ with scalar curvature a constant function. Of course, the presence of the critical Sobolev exponent made the resolution   of such equation  difficult and appealed to a more sophisticated analysis. We can refer the reader to the book \cite{Hebey} for a compendium on this topic. Equation \eqref{eq1.1} can be then seen as  a Yamabe type equation of a singular type.\\
 When the function $\rho_p$ is of power $0<\gamma<2$, the study of the associated equations is related to the study of conformal deformation to constant scalar curvature of metrics which are smooth only in some geodesic ball $B(p,\delta)$ (see  \cite{Madani} ). As the inclusion $H^2_1(M)\subset L(M,\rho_p^{-\gamma})$ (where  $H^2_1(M) $ and  $L(M,\rho_p^{-\gamma})$ are  defined in section 2)   is compact for $0<\gamma<2$, the study of existence of solutions, in this case, goes as in the case of  'regular' Yamabe equation  (see  \cite{Madani} ).\\
However, when $\gamma=2$,  regarding the non compactness of the inclusion $H^2_1(M)\subset L(M,\rho_p^{-2})$,  equation \eqref{eq1.1} is also critical in terms of the power,$\gamma=2 $,  of the function $\rho_p$.\\
In studying equations \eqref{eq1.1}, besides the critical Sobolev exponent $2^*$, the singular term plays a prominent role. As it has been shown in \cite{Maliki},  it interferes in the decomposition of the Palais-Smale sequence of the functional energy and then collaborates principally in determining the safe energy level for the compactness of the Palais-Smale sequences.\\
The singular term interferes also in the regularity of solutions in that only weak solutions can be obtained as contrasted to the case of the 'regular' Yamabe equation where strong solutions can be obtained (see \cite{Madani}).
The author in \cite{Madani} studied equation \eqref{eq1.1} and proved the existence of at least one solution of minimal energy. In this work, we are interested in the existence of multiple solutions of high energy. The main tool that we employ to achieve our interest is the  classical Lusternik-Schnirelmann theory (see for example \cite{Amrosetti}). We note that multiple solutions of minimal energy can also be obtained.\\
 In \cite{Benci}, the authors proved a multiplicity result for a subcritical regular equation on compact Riemannian manifold. They, used Lusternik-Schnirelmann theory together with some astute constructions. We will follow the authors  in \cite{Benci} and \cite{Chabrowski}  in their global framework. As aforesaid, equation \eqref{eq1.1} is double critical, which leads to further technical difficulties to arise and then a deeper analysis needs to be done. \\
 The paper is organized as follows: in section 2 we introduce some notations and useful results that will be of great use and state the main result. In section 3, a noncompact analysis is done. In section 4, we give an overview of the proof of the main result and then collect  ingredients for the proof of the main results. The fourth section is devoted to the proof of the main result.
\section{Notations, useful results and statement of the main result.}
In this section, we introduce some notations and cite results that are useful
in our study.\\
Throughout the paper, we will denote by $B(a,r)$ a ball of center
$a$ and radius $r>0$, the point $a$ will be specified either in $M$
or in $\R^n$, and $B(r)$ is a ball in $\R^n$ of center $0$ and
radius $r>0$.\\
Let $q\in M$. Denote by $\exp_q$ the exponential map which defines, for $r>0$ small, a diffeomorphism from $ B(r)$ to $ B(q,r)$.\\ Let $H_{1}^{2}(M)$ be the Sobolev space consisting of
the completion of $\C^{\infty}(M)$ with respect to the norm
$$||u||_{H_{1}^{2}(M)}=\int_M(|\nabla u|^2+u^2)dv_g.$$
$M$ being compact, $H_{1}^{2}(M)$ is then  embedded in $L_q(M)$
compactly for $q<2^*=\frac{2n}{n-2}$ and continuously for $q=2^*$.\\
Let $K(n,2)$ denote the best constant in Sobolev inequality that
asserts that there exists a constant $B>0$ such that for any $u\in
H_{1}^{2}(M)$,
\begin{equation}\label{2.1}
 ||u||^2_{L_{2^*}(M)}\le K^2(n,2)||\nabla u||^2_{L_2(M)}+B||u||^2_{L_2(M)}.
\end{equation}
The constant $K(n,2)$ is defined to be
\begin{equation*}
K(n,2)=\inf_{u\in  H^1(\mathbb{R}^n)\setminus{0})} \frac{\int_{\mathbb{R}^n}|\nabla u|^2dx}{(\int_{\mathbb{R}^n}|u|^{2^*})^{\frac{2}{2^*}}}.
\end{equation*}
It is well known that the extremal functions for the above infinimum are the family of functions
   \begin{equation} \label{famr}
w_{\mu}(x)=(n(n-2))^{\frac{n-2}{4}}\left(\frac{\mu}{\mu^{2}+|x|^{2}}\right)^{\frac{n-2}{2}}, \mu>0.
    \end{equation}
These family of functions classifies all positive solutions of the Euclidean equation
\begin{equation}\label{eqr}
    \Delta u= u^{2^*-1}.
\end{equation}
Denote by $L_2(M,\rho_p^{-2})$ the space of
functions $u$ such that $\frac{u^2}{\rho_p^{2}}$ is integrable.
This space is endowed with norm
$\|u\|^2_{2,\rho_p^{-2}}=\int_M\frac{|u|^2}{\rho_p^{2}}dv_g$.\\
In \cite{Madani}, the author proved the  following Hardy inequality: let  $(M,g)$ be any compact manifold $M$,  for every $\varepsilon>0$ there exists a
positive constant $A(\varepsilon)$ such that for any $u\in
H^2_1(M)$,
\begin{equation}\label{2.3}
\int_M \frac{u^2}{\rho^2_p}dv_g\le(
K^2(n,-2)+\varepsilon)\int_M|\nabla
u|^2dv_g+A(\varepsilon)\int_Mu^2dv_g,
\end{equation}
 with $K(n,-2)$ being the
best constant in the Euclidean Hardy inequality
\begin{equation*}
\int_{\mathbb{R}^n} \frac{u^2}{|x|^2}dx \le
K(n,-2)^2\int_{\mathbb{R}^n}|\nabla u|^2dx,
u\in\C^\infty_o(\mathbb{R}^n).
\end{equation*}
 The constant $K(n,-2)$ is equal to $ \frac{2}{n-2}$ and is not attained.\\
 If $u$ is supported in some ball $B(p,\delta),0<\delta<\delta_g$,  then there exists positive constant $K_\delta(n,-2)$
\begin{equation}\label{2.4}
\int_{B(p,\delta)} \frac{u^2}{\rho^2_p}dv_g\le
K_\delta(n,-2)\int_{B(p,\delta)}|\nabla u|^2dv_g,
\end{equation}
with $ K_\delta(n,-2)$ goes to $K(n,-2)$ as $\delta$ goes to$0$. \\
 On the Euclidean space $\mathbb{R}^n$, the author in \cite{Terracini}  studied the equation
\begin{equation}\label{eqs}\tag{$Eu_\lambda$}
\Delta u -\frac{\lambda}{|x|^2}=|u|^{\frac{4}{n-2}}u,
\end{equation}
where $\lambda>0$ is a positive constant. She proved in particular that for $\lambda\ge \frac{(n-2)^2}{4}$ there is no positive solution and for $0<\lambda<K^2(n,-2)=\frac{(n-2)^2}{4} $, all positive solutions are the class of functions
 \begin{equation}\label{fams}
w_{\lambda,\xi}(x)=(n(n-2)a_\lambda^2)^{\frac{n-2}{4}}
\left(\frac{\xi^{a_\lambda}|x|^{a_\lambda-1}}{\xi^{2a_\lambda}+|x|^{2a_\lambda}}\right)^{\frac{n-2}{2}},\xi>0,
\end{equation}
 where $a_\lambda= \sqrt{1-\lambda K^2(n,-2)}$.
Note that  for $\lambda=0$, we meet the functions \eqref{fams}. Furthermore, if we denote by $S$ the infimum
 \begin{equation*}
S=\inf_{u\in D^{1,2},u\neq0}\frac{(\int_{\mathbb{R}^n}|\nabla u|^2-\lambda\frac{u^2}{|x|^2})dx}{(\int_{\mathbb{R}^n}|u|^{2^*}dx)^{\frac{2}{2^*}}}
 \end{equation*}
 then, the functions defined by \eqref{fams} are extremal for this infimum , that is
 \begin{equation}\label{2.6}
S_\lambda=\frac{(\int_{\mathbb{R}^n}|\nabla w_{\lambda,\xi}|^2-\lambda\frac{w_{\lambda,\xi}^2}{|x|^2})dx}
{(\int_{\mathbb{R}^n}|w_{\lambda,\xi}|^{2^*}dx)^{\frac{2}{2^*}}}.
 \end{equation}
Moreover,
 \begin{equation}\label{2.7}
S_\lambda=\frac{(1-\lambda K^2(n,-2))^{\frac{n-1}{n}}}{K^2(n,2)}.
 \end{equation}
 Let $h$ be  a continuous function on $M$ and $p\in M$ a fixed point. Let us take $\lambda = h(p)>0$ with $1-h(p) K(n,-2)^2>0$. We denote by $D^*$ the constant
 \begin{eqnarray}\label{2.8}
    D^*=\frac{(S_{h(p)})^{\frac{n}{2}}}{n}=\frac{(1-h(p) K^2(n,-2))^{\frac{n-1}{2}}}{n K^2(n,2)}.
    \end{eqnarray}
Let
\begin{equation*}
\mu=\inf_{u\in H^2_1(M),u\neq0}\frac{\int_M(|\nabla
u|^2-\frac{h}{\rho_p^2}u^2)dv_g}{(\int_M|u|^{2^*}dv_g)^{\frac{2}{2^*}}}.
\end{equation*}
In \cite{Madani}, the author proved  an existence result for equation \eqref{eq1.1} on compact manifold under the condition
\begin{equation*}
\mu<\frac{1-h(p)K^{2}(n,-2)}{K^{2}(n,2)}=(nD^*)^{\frac{2}{n}}(1-h(p)K^{2}(n,-2))^{\frac{1}{n}}.
\end{equation*}
provided of course that $1-h(p)K^{2}(n,-2)>0$.
Further in this paper, we will see that this existence result is a simple consequence of a Palais-Smale decomposition result that we established in \cite{Maliki}. Furthermore, we consider the reverse inequality $\mu>( nD^*)^{\frac{2}{n}}$ and  show that effectively multiple solutions exist in this case.  In a very precise way, we establish the following result
\begin{Theorem} Let $(M,g)$ be a compact Riemannian manifold of dimension $n$.  Suppose that the function $h$ is smooth, changes sign once and satisfies the following conditions
\begin{enumerate}
\item $h(p)>0, 0<1-h(p)K(n,-2)^2<\frac{1}{2}$,
\item $n=Dim(M)> \frac{2}{a}+2, a=\sqrt{1-h(p)K^2(n,-2)}$,
\item $(A(n,a)+h(p))Scal_g(p)-6\frac{\Delta h(p)}{n}>0$, where $A(n,a)$ is defined by \eqref{4.16}.
\end{enumerate}
If $\mu>( nD^*)^{\frac{2}{n}}$, equation \eqref{eq1.1} admits at least $Cat(M)$ ( $ Cat(M)$ is the Lusternik-Schnirelmann category of $M$ defined in section 3)  weak solutions such that each weak solution $u$ satisfies $D^*<J_h(u)<D^*+ g(\varepsilon)$ and at least one weak solution $u$ such that $D^*+ g(\varepsilon)< J_h(u)$.
\end{Theorem}
\section{compactness of Palais-Smale sequences}
Let $J_h$ be the functional defined on $H_1^2(M)$ by
\begin{equation*}
J_h(u)=\frac{1}{2}\int_M(|\nabla
u|^2-\frac{h}{\rho^2_p}
u^2)dv_g-\frac{1}{2^*}\int_M|u|^{2^*}dv_g.
\end{equation*}
It is a $C^2$ functional whose critical points are weak solutions of equation \eqref{eq1.1}. \\
A Palais-Smale sequence $u_m$ ( P-S in short ) of $J_h $ at a level
$d$ is defined to be the sequence that satisfies $J_h(u_m)\to
d$ and $DJ_h(u_m)\varphi\to 0,\forall\varphi \in H^2_1(M)$.\\
 The functional $J_h$is said to satisfy P-S condition et level $d$ if each P-S sequence at level $d$ is relatively compact.\\
In this section,  we  determine a region of levels where P-S condition is satisfied and then  critical points of the function $J_h$ can be obtained.This can be done by analyzing asymptotically the behavior of the P-S sequences. In a previous paper \cite{Maliki}, based on blow-up theory in \cite{Druet-hebbey-robert} and on a result in \cite{D.Smet}, we asymptotically studied P-S sequences of the functional $J_h$ and established  a Struwe-type decomposition formula for P-S sequences of the functional $J_h$. For the seek of clearness we cite this theorem and we refer to \cite{Maliki} for a detailed proof. Let us first introduce on $D^{1,2}(\R^n)$ the   functionals
\begin{eqnarray*}
J(u)&=& \frac{1}{2}\int_{\R^n}|\nabla
u|^2dx-\frac{1}{2^*}\int_{\R^n}|u|^{2^*}dx, \text{ and } \\
 J_\infty (u)&=& \frac{1}{2}\int_{\R^n}|\nabla
u|^2dx- \frac{h(p)}{2}\int_{\R^n}\frac{u^2}{|x|^2}
dx-\frac{1}{2^*}\int_{\R^n}|u|^{2^*}dx.
\end{eqnarray*}
In \cite{Maliki}, we established the following decomposition  theorem:
\begin{Theorem}\label{thm3.6} Let $(M,g)$ be a compact Riemannian manifold with
$dim(M)=n\ge3$ and let $h$ be a continuous function on $M$ that  on the point $p\in M$, it satisfies $0<h(p)<\frac{1}{K(n,2,-2)^2}$.\\
Let $u_m$ be a P-S sequence of the functional $J_h$ at level $d$. Then, there
exist $k \in \N$, sequences $ R_m^i>0,R_m^i\underset{m\to\infty}{\to}0$, $\ell\in\N^n$
sequences $\t_m^j>0,\t_m^j\underset{m\to\infty}{\to}0$,
converging sequences  $x_m^j\to x_o^j\neq p$ in $M$, a solution
$u\in H^2_1(M)$ of \eqref{eq1.1},  solutions $v_i\in D^{1,2}(\R^n)$
 of $ Eu_{h(p)}$ and nontrivial solutions $\nu_j\in
D^{1,2}(\R^n)$ of \eqref{eqr} such that up to a subsequence
\begin{eqnarray}\label{3.9}
\nonumber_m&=&u+\sum_{i=1}^{k}(R^i_m)^{\frac{2-n}{n}}
\eta_r(\exp^{-1}_p(x))v_i((R_m^i)^{-1}\exp^{-1}_p(x))\\&+&\sum_{j=1}^{\ell}(\tau^i_m)^{\frac{2-n}{n}}
\eta_r(\exp^{-1}_{x_m^j}(x))\nu_j((\tau_m^j)^{-1}\exp^{-1}_{x_m^j}(x))+\W_m,\\&&\nonumber
\text{ with } \W_m\to 0 \text{ in }H^1_2(M),
\end{eqnarray}
and
\begin{equation}\label{3.10}
J_h(u_m)=J_h(u)+\sum_{i=1}^k
J_\infty(v_i)+\sum_{j=1}^l J(\nu_j)+o(1).
\end{equation}
\end{Theorem}
Before we derive some consequences of the above theorem we draw attention to the following important remark
 \begin{remark} If $u$ is a changing sign of equation \eqref{eqs} with $\lambda=h(p)$, then $J(u)> 2D^*$.
 \end{remark}
  In fact write $u=u^++u^-$, where $u^+=\max(u,0)$ and $v^-=\min(u,0)$. We then get
\begin{eqnarray}\label{3.6}
\nonumber\int_{\R^n}\left(|\nabla u^+|^2-\frac{h(p)}{|x|^2}{u^{+}}^{2}\right)dx&=&\int_{\R^n}(\nabla u.\nabla u^+-\frac{h(p)}{|x|^2}uu^{+})dx\\&=&\int_{\R^n}|u|^{2^*-1}uu^+dx=\int_{\R^n}|u^+|^{2^*}dx
\end{eqnarray}
Since $u^+$ cannot be a 'member' of the family of functions defined by \eqref{fams},  then by \eqref{3.6} we get
\begin{eqnarray*}
J_\infty(u^+)&=&\frac{1}{n}\int_{\R^n}\left(|\nabla u^+|^2-\frac{h(p)}{|x|^2}{u^{+}}^{2}\right)dx\\&>&\frac{1}{n}(S_{h(p)})^{\frac{n}{2}}=D^*,
\end{eqnarray*}
where $S_{h(p)}$ is defined by \eqref{2.7}.\\
By the same way, we get
\begin{equation*}
J_\infty(u^-)=\frac{1}{n}\int_{\R^n}\left(|\nabla u^-|^2-\frac{h(p)}{|x|^2}{u^{-}}^{2}\right)dx>D^*
\end{equation*}
Thus, we obtain
\begin{equation*}
    J_\infty(u)=J_\infty(u^+)+J_\infty(u^-)>2D^*.
\end{equation*}
Now, we derive  from the above theorem the following corollaries
 \begin{Corollary} \label{cor1} Under conditions
 \begin{enumerate}
   \item $\mu\ge (nD^*)^{\frac{2}{n}}$,
   \item $0<1-h(p)K^2(n,-2)\le\frac{1}{2}$,
 \end{enumerate}
  every P-S sequence  of the functional $J_h$ at level $d$ with $D^*<d< 2D^*$ is relatively compact.
 \end{Corollary}
 \begin{proof} By the above theorem, there exist a critical point $u_o$ of $J_h$, a sequence of solutions $v_i$ of \eqref{eqs} and sequence of non trivial solutions $\nu_j$  such that up to a subsequence \eqref{3.9} and \eqref{3.10} hold. Suppose that $v_i\neq0$ for some $i$, either $v_i$ changes sign or  not, it must hold $d>2D^*$. Thus $v_i=0 \forall i$. Similarly, if there exists $\nu_j\neq0$, by condition (1) of the corollary will have also $ d>2D^*$. Therefore, all $\nu_j$ are null and thus $u_m$ converges strongly up to a subsequence in $H^2_1(M)$.
 \end{proof}
\begin{Corollary} \label{cor2} Suppose that  $\mu>(nD^*)^{\frac{2}{n}}$. Then,  for every  P-S sequence $u_m$ of $J_h$ at level $ D^*$, there exists a sequence of functions $w_m\in H^2_1(M)$ such that $w_n\to 0$ strongly in $H^2_1(M)$  and
\begin{equation*}
    u_m=w_m+\phi_{p,R_m\varrho},\varrho>0,
\end{equation*}
 where $\phi_{p,R_m\varrho}$ is the function defined by \eqref{4.23} with $q=p$ and $\xi= R_m\varrho$.
 \end{Corollary}
 \begin{proof} First, the condition $\mu>(nD^*)^{\frac{2}{n}}$ prevents the existence of any critical point $u_o$ of $J_h$ with  $J_h(u_o)= D^*$. Thus, only one function  $v_i $ can be included in the decomposition expression of the above theorem and since this function cannot change sign, then it takes the form of \eqref{4.23} with $q=p$ and $\xi= R_m\varrho$.
 \end{proof}
\begin{Corollary} under   the following conditions
\begin{enumerate}
\item $1-h(p)K^2(n,-2)>0$,
\item $\mu< (nD^*)^{\frac{2}{n}}$,
 \end{enumerate}
there  exists  a non trivial critical point of $J_h$.
\end{Corollary}
\begin{proof} Like in corollary \ref{cor1}, by applying theorem \ref{thm3.6}, it is not difficult to see that the P-S condition for the functional $J_h$ is satisfied for any level $d$ such that  $0<d<D^*$.\\
Consider $d=\inf_{\NN_h}J_h$, where $N_h$ is the Nehari manifold defined by \eqref{4.12'}. By applying the Ekland variational principle, we can obtain a P-S sequence on $\NN_h$  at level $d$ which is also a P-S sequence $u_m$ on $H^2_1(M)$.  It is clear that $d\ge\frac{1}{n} \mu ^{\frac{n}{2}}$. Let $u\in H^{2}_1(M)\setminus\{0\}$, then $\Phi(u)u\in \NN_h$, where $\Phi$ is defined by \eqref{4.17'}, and by homogeneity of
\begin{equation*}
I_h(u)=\frac{\int_M(|\nabla
u|^2-\frac{h}{\rho_p^2}u^2)dv_g}{(\int_M|u|^{2^*}dv_g)^{\frac{2}{2^*}}},
\end{equation*}
 since $\Phi(u)u\in \NN_h$, we get that
\begin{equation*}
I_h(u)=I_h(\Phi(u)u)=(nJ_h(\Phi(u)u))^{\frac{2}{n}}\ge (nd)^{\frac{2}{n}}.
\end{equation*}
Thus we get $\mu\ge (nd)^{\frac{2}{n}}$ and hence $ d=\frac{1}{n} \mu ^{\frac{n}{2}}$.Therfore, condition (2) of the corollary implies that $d<D^*$ and hence the P-S sequence $u_m$ converges up to a subsequence strongly in $H^2_1(M)$ to a critical point of $J_h$.
\end{proof}
\section{Construction of solutions }
In this section, we construct solutions of \eqref{eq1.1}as  critical points of the functional $J_h$. In searching critical points of the functional $J_h$, we just apply the following classical theorem.
 \begin{Theorem}\label{thm-L-S} Let $ J $ be  $ C^1$ real functional defined on a $ C^{1,1}$ Banach manifold $ N$. If $J$ is bounded from below on $N$ and satisfies the P-S condition then it has at least $Cat(J^c)$ critical points in $J^c$. \\
 Moreover, if $N$ is contractible  and $ Cat(J^c)>1$ then there exists at least one critical point $u\notin J^c$.
\end{Theorem}
 $J^c$ in the theorem denotes  the sub-level set of the functional $ J $
\begin{equation*}
    J^c=\{\ u\in N: J(u)<c\}.
\end{equation*}
 and $Cat(J^c)$ denotes the  Lusternik-Schnirelmann category of the set $ J^c$.\\
 We recall that the Lusternik-Schnirelmann category  $ Cat_Y(X) $ of a topological space $X$ with respect to a topological space $Y$ with  $X\subset Y$ is the least  integer $k\le\infty$  such that there exists an open covering of $U_i$ of $X$ with  each $U_i$  contractible in $Y$. If $X=Y$, we put $Cat_X(X)=Cat(X)$.\\
Consider the Nehari manifold $\NN_h$ which associated to the functional $J_h$
 \begin{equation}\label{4.12'}
    \NN_h=\{u\in  H^2_1(M)\setminus\{0\}, DJ_h(u).u=0\}.
 \end{equation}
It is well known that this manifold defines a natural constraint set for the functional $J_h$ in the sense that a P-S sequence in $\NN_h $ is also a P-S of $J_h$ on $ H^2_1(M)$. Moreover, for $u \in H^2_1(M)\setminus\{0\}$, we have $\sup_{t>0}(tu)=t_ou$ with  $t_o =\left(\frac{\int_M(|\nabla u|^2-\frac{h}{\rho_p^2}u^2)dv_g}{(\int_M|u|^{2^*}dv_g)^{\frac{2}{2^*}}}\right)^{\frac{n-2}{4}}$ and $t_ou\in \NN_h.$\\
Note that if the function $h$ changes sign only once and $1-h(p)K(n,-2)^2>0$, then $ \int_M(|\nabla u|^2-\frac{h}{\rho_p^2}u^2)dv_g>0$. In  fat, let $\delta_0=\max \delta$ such that $h(x)\ge0 $ on $B(p,\delta)$. Without loss of generality we assume that $h(p)=\max_M(h)$, by inequality \eqref{2.4} we have
\begin{eqnarray*}
\int_M(|\nabla u|^2-\frac{h}{\rho_p^2}u^2)dv_g&\ge& \int_M|\nabla u|^2dv_g-\max_Mh\int_{B(p,\delta)}\frac{u^2}{\rho_p^2}dv_g\\&\ge& (1-h(p)K_\delta(n,-2)^2)\int_M|\nabla u|^2dv_g.
\end{eqnarray*}
Since $K_\delta(n,-2)\to K(n,-2)$ as $\delta \to0$ and $1-h(p)K(n,-2)^2>0$, we get the claim true.\\
The main difficulty in applying  theorem \ref{thm-L-S} above is that the P-S condition for the functional $J_h$ is not satisfied for any level because of the presence of the critical exponent $2^*$  and  the critical singular term. However, corollary \ref{cor1} gives  level rank for which P-S condition is satisfied and consequently the P-S sequence levels  would  be restricted to this level rank. We therefore construct a subset of the manifold $\NN_h$ on which the P-S condition is satisfied. It seems that the 'test' functions defined by \eqref{4.17} play an important role in the construction of such subset.\\
We can assume by the Nash embedding theorem, without loss of generality, that the Riemannian manifold $ M$ is embedded in some Euclidean space $\mathbb{R}^N$.\\
Let  $M_{ r}$ be the set
\begin{equation*}
    M_{r}=\{ x\in \mathbb{R}^N: d(x,M)<r\},
\end{equation*}
and define the radius of the topological invariance $r_M$ of $M$ by
\begin{equation*}
    r_M=\sup \{ r>0: Cat(M_r)=Cat(M)\}.
\end{equation*}
For $\varepsilon>0$, let $ g(\varepsilon) $  be a positive function such that $g(\varepsilon)\to 0$ as $\varepsilon\to 0$. Let $\Sigma_\varepsilon$ be the subset of $\NN_h$ defined by
\begin{equation*}
\Sigma_\varepsilon=\{u\in  \NN_h \text{ s.t }  D^*<J_h(u)<D^*+g(\varepsilon) \text{ for some } g(\varepsilon)\}.
\end{equation*}
To prove the main theorem,  we construct to continuous maps $\I_\varepsilon: M\rightarrow \Sigma_\varepsilon$ and $\beta: \Sigma_\varepsilon\rightarrow M_{r_M}$ such that the composition $\beta o\I_\varepsilon$ is homotopic to the identity. This leads, by the Lusternik-Schnirelmann properties (see \cite{Amrosetti} for example) that $ Cat(\Sigma_\varepsilon)>Cat(M)$. Thus by applying theorem \ref{thm-L-S} on the set $\Sigma_\varepsilon$, we obtain at least $Cat(M)$ critical points of the functional $J_h$ in $\Sigma_\varepsilon$. Finally, we end the proof of main theorem by proving the existence on another critical point $u\notin \Sigma_\varepsilon$. This can be done by constructing a contractible set $P_\varepsilon$ that contains $\I_\varepsilon(M)$ and is contractible in $\NN_h\cap J^{C_\varepsilon}$ for  bounded $C_\varepsilon$. \\
 First, we have to prove  that the set $ \Sigma_\varepsilon $ is not empty. This is achieved in lemma \eqref{Lem-expan} below.\\
Let $q\in M$ be any point of $M$ and $0<\delta<\frac{\delta_g}{2}$. Define a cut-off function  on $M$, $\eta_{q,\delta} $,  such that $0\le\eta_{q,\delta}\le1$ , $\eta_{q,\delta}(x)=1, x\in B(q,\delta), \eta_{q,\delta}(x)=0, x\in M\setminus\in B(q,2\delta)$ and $|\nabla\eta_{q,\delta}|\le C$, for some constant $C>0$.\\
 Put $\rho_p(x)=r$ and  consider  on $M$ the function
 \begin{equation*}
    \phi_\varepsilon(x)= C(n,a)
    \eta_{p,\delta}\left(\frac{\varepsilon^{a}r^{a-1}}{\varepsilon^{2a}+r^{2a}}\right)^{\frac{n}{2}-1},
\end{equation*}
where
\begin{equation}\label{4.12}
 C(n,a)= (a^2n(n-2))^{\frac{n-2}{4}} \text{ and  } a=\sqrt{1-h(p)K(n,2,-2)^2}
\end{equation}
Define the constants
\begin{eqnarray}\label{4.13}
\nonumber C_1(n,a)&=&\frac{1}{6}C(n,a)^2(\frac{n-2}{2})^2w_{n-1}[(a-1)^2+2(1-a)\frac{a(n-2)+2}{an-2}
\\  &+&(1+a)^2\frac{(an+2)(a(n-2)+2)}{(an-2)(a(n-2)-2)}]
\end{eqnarray}.
\begin{equation}\label{4.14}
C_2(n,a)=C(n,a)^2\frac{4a^2w_{n-1}(n-2)(n-1)}{(a(n-2)-2)(an-2)}
\end{equation}
\begin{equation}\label{4.15}
C_3(n,a)=C(n,a)^{2^*}\frac{w_{n-1}(a(n-2)+2)n}{6(an-2)}.
\end{equation}
Let the constant
\begin{equation}\label{4.16}
     A(n,a)= \frac{6(\frac{n-2}{n}C_3(n,a)-C_1(n,a))}{C_2(n,a)}
\end{equation}
Consider the projection $\Phi:H^2_1(M)\setminus\{0\} \rightarrow \NN_h$ defined by
 \begin{equation} \label{4.17'}
    \Phi(u)= \left(\frac{\int_M(|\nabla u|^2-\frac{h}{\rho_p^2}u^2)dv_g}{(\int_M|u|^{2^*}dv_g)^{\frac{2}{2^*}}}\right)^{\frac{n-2}{4}}u
 \end{equation}
  In the remaining of the paper, $\alpha(\varepsilon)$ is a function such that $\alpha(\varepsilon)>0$ and $ \alpha(\varepsilon)\to0$ as $\varepsilon\to 0$.\\
In \cite{Maliki}, we have proved the following lemma. For completeness, we review briefly the proof
\begin{Lemma} \label{Lem-expan}Suppose that
\begin{eqnarray*}
   \label{dim1}n=dim(M)>2+\frac{2}{a},&& \text{ and }  \\
   \label{dim2}(A(n,a)+h(p))Scal_g(p)-6\frac{\Delta h(p)}{n}>0.
\end{eqnarray*}
Then, there exists $ g(\varepsilon)>0$ with $g(\varepsilon)\to 0 $ as $\varepsilon\to 0$ such that
\begin{equation}\label{4.17}
 D^*<J_h(\Phi(\phi_\varepsilon))<D^*+g(\varepsilon), g(\varepsilon)>0, \forall \varepsilon.
\end{equation}
\end{Lemma}
\begin{proof}
 Define for $2a\beta-1>\alpha>0$
 \begin{equation*}
   I^\alpha_{\beta}= \int_0^{\infty}\frac{r^{\alpha}}{(1+ r^{2a})^\beta}dr.
\end{equation*}
Then, by direct computations ( see \cite{Maliki})  one can get
\begin{eqnarray}\label{4.18}
&&\int_M|\nabla\phi_\varepsilon|^2dv_g\\ \nonumber &=&\int_{\R^n}|\nabla U|^2dx -Scal_g(p)C_1(n,a)
I_n^{a(n-2)+1}\varepsilon^2+o(\varepsilon^2)+\alpha(\varepsilon).
\end{eqnarray}
Similarly, by writing
\begin{equation*}
    h(x)=h(p)+(\nabla_ih)(p)x_i+(\nabla_{i,j}h)(p)x_ix_j+o(r^2)
\end{equation*}
  we obtain
\begin{eqnarray}\label{4.19}
&&\int_M\frac{h(x)}{r^2}\phi_\varepsilon^2dv_g\\ \nonumber &=& h(p)\int_{\R^N} \frac{U^2}{|x|^2}dx-[Scal_g(p)h(p)-\frac{\Delta h(p)}{n}]C_2(n,a)I^{a(n-2)+1}_{n}\varepsilon^2\\ \nonumber&+&o(\varepsilon^2)+\alpha(\varepsilon)],
\end{eqnarray}
For the term term $\int_M|\phi_\varepsilon|^{2^*}dv_g$, one can obtain
\begin{eqnarray}\label{4.20}
&&\int_M|\phi_\varepsilon|^{2^*}dv_g\\ \nonumber&=& \int_{\R^n}|U|^{2^*}dx-Scal_g(p)C_3(n,a)I_n^{a(n-2)+1}\varepsilon^2+o(\varepsilon^2)
\\&+& \nonumber\alpha_3(\varepsilon),
\end{eqnarray}
with $\lim_{\varepsilon\to0}\alpha(\varepsilon)=0$.\\
Using the fact that
\begin{equation*}
    \frac{\int_{\R^n}(|\nabla U|^2-h(p)\frac{U^2}{|x|^2})dx}{(\int_{\R^n}|U|^{2^*}dx)^{\frac{2}{2^*}}}=
    \frac{(1-h(p)K(n,2,-2)^2)^{\frac{n-1}{n}}}{K(n,2)^2}=(nD^*)^{\frac{2}{n}},
\end{equation*}
the expansions \eqref{4.18}, \eqref{4.19} and \eqref{4.20} yield
\begin{eqnarray*}
    &&\frac{ (\int_M|\nabla\phi_\varepsilon|^2-\frac{h}{r^2}\phi_\varepsilon^2)dv_g}
    {(\int_M|\phi_\varepsilon|^{2^*}dv_g)^{\frac{2}{2^*}}}\\&=&(nD^*)^{\frac{2}{n}}(1+
    \frac{1}{\int_{\R^n}|U|^{2^*}dx}[Scal_g(p)(\frac{n-2}{n}C_3(n,a)-C_1(n,a))\\&+&(\frac{Scal_g(p)h(p)}{6}-\frac{\Delta h(p)}{n})C_2(n,a))]I_n^{a(n-2)+1}\varepsilon^2+o(\varepsilon^2)+\alpha(\varepsilon).
\end{eqnarray*}
with $\lim_{\varepsilon\to0}\alpha(\varepsilon)=0$.\\
Now, writing
\begin{equation*}
    nJ(\Phi(\phi_\varepsilon))=\left(\frac{ (\int_M|\nabla\phi_\varepsilon|^2-\frac{h}{r^2}\phi_\varepsilon^2)dv_g}
    {(\int_M|\phi_\varepsilon|^{2^*}dv_g)^{\frac{2}{2^*}}}\right)^{\frac{n}{2}}
\end{equation*}
 we obtain
 \begin{eqnarray*}
     &&J_h(\Phi(\phi_\varepsilon))\\&=&D^*(1+
    \frac{n}{2(\int_{\R^n}|U|^{2^*}dx)^{\frac{n}{2}}}[Scal_g(p)((\frac{n-2}{n}C_3(n,a)-C_1(n,a))\\&+&(\frac{Scal_g(p)h(p)}{6}-\frac{\Delta h(p)}{n})C_2(n,a))]I_n^{a(n-2)+1}\varepsilon^2+o(\varepsilon^2)+\alpha(\varepsilon).
 \end{eqnarray*}
That is
 \begin{eqnarray}\label{4.21}
 &&\nonumber J_h(\Phi(\phi_\varepsilon))\\&=&D^*\left[1+
    B(n,a)\left((A(n,a)+h(p))Scal_g(p)-6\frac{\Delta h(p)}{n}\right)\varepsilon^2\right]\\ \nonumber&+&o(\varepsilon^2)+\alpha(\varepsilon)\nonumber.
 \end{eqnarray}
 with \begin{equation}\label{4.22}
 B(n,a)=\frac{n}{12C_2(n,a)(\int_{\R^n}|U|^{2^*}dx)^{\frac{n}{2}}}.
 \end{equation}
Therefore, if
 \begin{equation*}
(A(n,a)+h(p))Scal_g(p)-6\frac{\Delta h(p)}{n}>0,
 \end{equation*}
 for $\varepsilon $ small enough, we get \eqref{4.17}.
\end{proof}
\subsection{ The map $\I_\varepsilon$.} In this subsection, we construct a continuous map $\I_\varepsilon: M\rightarrow \Sigma_\varepsilon $.
For a fixed point $q\in M$ put $r_q(x)=dist_g(q,x), x\in M$ and let $\phi_{q,\xi} $ be the function
\begin{equation}\label{4.23}
    \phi_{q,\xi}(x)= C(n,a)
    \eta_{q,\delta}\left(\frac{\xi^{a}r_q(x)^{a-1}}{\xi^{2a}+r_q(x)^{2a}}\right)^{\frac{n}{2}-1},\xi>0,
\end{equation}
 where $a$ and $c(n,a)$ are defined by \eqref{4.12}.\\
For $\varepsilon\in (0,1)$, define the function $\mathcal{I_\varepsilon}:M\to \NN_h$ by
\begin{equation*}
\mathcal{I_\varepsilon}(q)=\Phi((1-\varepsilon)\phi_{p,\varepsilon}+\varepsilon\phi_{q,\varepsilon}).
\end{equation*}
Let us  prove the following lemma
\begin{Lemma}\label{lem-inj} The function $\mathcal{I}_\varepsilon$ is continuous, and under the conditions
\begin{enumerate}
  \item $ a(n-2)>2$,
\item $(A(n,a)+h(p))Scal_g(p)-6\frac{\Delta h(p)}{n}>0$,
\end{enumerate}
$\mathcal{I}(q)\in \Sigma_\varepsilon$ for all $q\in M$.
\end{Lemma}
\begin{proof} By continuity of the projection $\Phi: H^2_1(M)(u)\setminus\{0\}\to\NN_h$, in order to prove the continuity of the  function $\I_\varepsilon(q)$, we need to just  prove the continuity of the function $ \phi_{q,\varepsilon}$ with respect to $q$.\\
Let  $q_j$ be a sequence of points of $M$ that converges to $q$ and prove that
\begin{equation*}
     \phi_{q_j,\varepsilon}\to \phi_{q,\varepsilon}\text{ in } H^2_1(M) \text{ as } q_j\to q .
\end{equation*}
Put $A_j= B(q_j,2\delta)\cap B(q,2\delta)$. Since $q_j\to q$ there exist $j_o$ such that $A_j\neq\emptyset$ for all $j\ge j_o$. Then, for $q_j$ close to $q$ we have
\begin{eqnarray*}
&\int_{A_j}| \phi_{q_j,\varepsilon}(x)-\phi_{q,\varepsilon}(x)|^2dv_g&\\&=\int_{\exp_q^{-1}(A_j)}| (\phi_{q_j,\varepsilon}-\phi_{q,\varepsilon})(\exp_q(z))|^2\sqrt{|g_{\exp(z)}|}dz&\\
&= C(n,a)^2[\int_{\exp_q^{-1}(A_j)}\eta_{q,\delta}(\exp_q(z))^2
|(U_{q_j}-U_q)(\exp_q(z))|^2 \sqrt{|g_{\exp(z)}|}dz\\
&+\int_{\exp_q^{-1}(A_j)}U^2_{q_j}(\exp_q(z))|\eta_{q,\delta}(\exp_q(z))-\eta_{q_j,\delta}(\exp_q(z))|^2\sqrt{|g_{\exp(z)}|} dz&\\ &+2\int_{\exp_q^{-1}(A_j)}\eta_{q,\delta}U_{q_j,\varepsilon}|(\eta_{q,\delta}-\eta_{q_j,\delta})(\exp_q(z))|
&\\&|(U_{q_j}-U_q)(\exp_q(z))|^2 \sqrt{|g_{\exp(z)}|}dz]&
\end{eqnarray*}
where
\begin{equation*}
    U_{q,\varepsilon}(x)=\left(\frac{\varepsilon^{a}r_q(x)^{a-1}}{\varepsilon^{2a}+r_q(x)^{2a}}\right)^{\frac{n}{2}-1},q\in M.
\end{equation*}
Using the fact that $U_{q_j}\to u_q$ and $\eta _{q_j,\varepsilon}\to \eta_{q,\varepsilon}$ pointwise together with the boundedness of $ \int_{\exp_q^{-1}(A_j)}U^2_{q_j}(\exp_q(z))\sqrt{|g_{\exp(z)}|} dz$, we get that
\begin{equation*}
    \int_{A_j}| \phi_{q_j,\varepsilon}(x)-\phi_{q,\varepsilon}(x)|^2dv_g\to 0.
\end{equation*}
Of course, outside the set $A_j$, $ \int_{M\setminus A_j}| \phi_{q_j,\varepsilon}(x)-\phi_{q,\varepsilon}(x)|^2dv_g\to 0.$\\
Similarly, the same conclusion holds for $ \int_{M}| \nabla \phi_{q_j,\varepsilon}(x)-\nabla\phi_{q,\varepsilon}(x)|^2dv_g$.\\
Now, for the proof of second part of the lemma,  we begin with the case for $q=p$. In this case $ \I_\varepsilon(p)=\Phi(\phi_{p,\varepsilon})$ and then the conclusion follows by lemma \ref{Lem-expan}.\\
For  $q\neq p$,  let $\delta>0$  be small enough such that $B(q,2\delta)\cap B(p,2\delta)=\emptyset$. In this way, the functions $\phi_{p,\varepsilon}$ and $\phi_{q,\varepsilon}$  are of disjoint supports. Then, we have
\begin{eqnarray*}
&\int_M(|\nabla( (1-\varepsilon)\phi_{p,\varepsilon}+\varepsilon\phi_{q,\varepsilon} )|^2-\frac{h}{r_p^2}((1-\varepsilon)\phi_{p,\varepsilon}+\varepsilon\phi_{q,\varepsilon} )^2)dv_g&\\&=(1-\varepsilon)^2\left(\int_M(|\nabla\phi_{p,\varepsilon}|^2-\frac{h}{r_p^2}
\phi_{p,\varepsilon})dv_g\right)+\varepsilon^2\left(\int_M(|\nabla\phi_{q,\varepsilon}|^2-\frac{h}{r_p^2}
\phi_{p,\varepsilon})dv_g\right)&
\end{eqnarray*}
We point out that by considering a normal geodesic coordinate system around the point $q$, the expansion \eqref{4.18} remains the same for any point $q$, that is
\begin{eqnarray*}
&&\int_M|\nabla\phi_{q,\varepsilon}|^2dv_g\\ \nonumber &=&\int_{\R^n}|\nabla U|^2dx -Scal_g(q)C_1(n,a)
I_n^{a(n-2)+1}\varepsilon^2+o(\varepsilon^2)+\alpha(\varepsilon).
\end{eqnarray*}
Moreover, we have
\begin{eqnarray*}
&&\int_M\frac{|h(x)|}{r_p^2}\phi_{q,\varepsilon}^2dv_g\le \sup_M|h(x|[\frac{\delta^{-2}}{4}\int_{B(q,2\delta)} \phi^2_{q,\varepsilon}(x) dv_g\\&&+C(n,a)
    \left(\frac{\varepsilon^{a}\delta^{a-1}}{\varepsilon^{2a}+\delta^{2a}}\right)^{n-2}\int_{M\setminus B(q,2\delta)}\frac{1}{r_p^2}dv_g.
\end{eqnarray*}
The second integral is bounded by the Hardy inequality. For the first integral, as in \cite{Maliki}, by considering a geodesic normal coordinate system around the point $q$,  direct calculations give
\begin{equation*}
\int_{B(q,2\delta)} \phi^2_{q,\varepsilon}(x) dv_g=C(n,a)^2w_{n-1}\varepsilon^2\int_0^\infty\frac{t^{a(n-1)+1}}{(1+t^{2a})^{n-2}}dt+ \alpha(\varepsilon).
\end{equation*}
  with $ w_{n-1}$ is the volume of the unit sphere $S^{n-1}\subset\mathbb{R}^n$. Since $a(n-2)> 2$, we get that
\begin{equation*}
   \int_M\frac{|h(x)|}{r_p^2}\phi_{q,\varepsilon}^2dv_g=O(\varepsilon^2)+o(\varepsilon^2).
\end{equation*}
Hence,  we obtain
\begin{eqnarray*}
&(\int_M|\nabla( (1-\varepsilon)\phi_{p,\varepsilon}+\varepsilon\phi_{q,\varepsilon} )|^2-\frac{h}{r_p^2}((1-\varepsilon)\phi_{p,\varepsilon}+\varepsilon\phi_{q,\varepsilon} )^2)dv_g&\\&=(1-\varepsilon)^2\left(\int_{\mathbb{R}^n}(|\nabla U|^2-\frac{h(p)}{|x|^2})
dx+ \alpha(n,a)\varepsilon^2\right)&\\&+\varepsilon^2\int_{\mathbb{R}^n}(|\nabla U|^2dx +o(\varepsilon)^2+\alpha(\varepsilon).&
\end{eqnarray*}
with
\begin{equation*}
    \alpha(n,a)= \left(Scal_g(p)C_1(n,a)-(\frac{Scal_g(p)h(p)}{6}-\frac{\Delta h(p)}{n})C_2(n,a)\right)I_n^{a(n-2)+1}
\end{equation*}
On the other hand,  since the functions $\phi_{p,\varepsilon}, \phi_{q,\varepsilon} $ are of disjoint supports, we have
\begin{eqnarray*}
  &\left(\int_M|(1-\varepsilon)\phi_{p,\varepsilon}+\varepsilon\phi_{q,\varepsilon}|^{2^*}dv_g\right)^{\frac{2}{2^*}} &\\&=  \left((1-\varepsilon)^{2^*}\int_M|\phi_{p,\varepsilon}|^{2^*}dv_g+ \varepsilon^{2^*}\int_M|\phi_{q,\varepsilon}|^{2^*}dv_g\right)^{\frac{2}{2^*}}.&
\end{eqnarray*}
Here again the expansion \eqref{4.20} holds true for $\int_M|\phi_{q,\varepsilon}|^{2^*}dv_g$. That is
\begin{equation*}
\int_M|\phi_{q,\varepsilon}|^{2^*}dv_g= \int_{ \mathbb{R}^n}|u|^{2^*}dx-Scal_g(q)C_3(n,a)I_n^{a(n-2)+1}\varepsilon^2+o(\varepsilon^2)
+\alpha(\varepsilon).
\end{equation*}
 Then
\begin{eqnarray*}
&\left(\int_M|(1-\varepsilon)\phi_{p,\varepsilon}+\varepsilon\phi_{q,\varepsilon}|^{2^*}dv_g\right)^{\frac{2}{2^*}} &\\&=((1-\varepsilon)^{2^*}+\varepsilon^{2^*})^{\frac{2}{2^*}}\left(\int_{\mathbb{R}^n}|U|^{2^*}dx-
\frac{(1-\varepsilon)^{2^*}}{(1-\varepsilon)^{2^*}+\varepsilon^{2^*}}(Scal_g(p)C_1(n,a)
I_n^{a(n-2)+1}\varepsilon^2\right.&\\& \left.+o(\varepsilon^2)+\alpha(\varepsilon))\right)^{\frac{2}{2^*}}&
\end{eqnarray*}
By using the expansions
\begin{equation*}
  \frac{(1-\varepsilon)^{2^*}}{(1-\varepsilon)^{2^*}+\varepsilon^{2^*}}=1+o(\varepsilon^2)
\end{equation*}
and
\begin{eqnarray*}
((1-\varepsilon)^{2^*}+\varepsilon^{2^*})^{\frac{2}{2^*}}&=&
(1-\varepsilon)^{2}(1+(\frac{\varepsilon}{1-\varepsilon})^{2^*})^{\frac{2}{2^*}}
\\&=&(1-\varepsilon)^{2}\left(1+ \frac{2}{2^*}(\frac{\varepsilon}{1-\varepsilon})^{2^*}
+o((\frac{\varepsilon}{1-\varepsilon})^{2^*})\right),
\end{eqnarray*}
and remark that $2^*>2$, we get
\begin{eqnarray*}
  &\left(\int_M|(1-\varepsilon)\phi_{p,\varepsilon}+\varepsilon\phi_{q,\varepsilon}|^{2^*}dv_g\right)^{\frac{2}{2^*}} &\\ \nonumber&=((1-\varepsilon)^2(\int_{\mathbb{R}^n}|U|^{2^*}dx)^{\frac{2}{2^*}}(1-\frac{2}{2^*\int_{\mathbb{R}^n}|U|^{2^*}dx}Scal_g(p)C_1(n,a)
I_n^{a(n-2)+1}\varepsilon^2&\\ \nonumber &+o(\varepsilon^2)+\alpha(\varepsilon))&
\end{eqnarray*}
Thus, we get the development of
 \begin{eqnarray*}
&&\left( nJ((1-\varepsilon)\phi_{p,\varepsilon}+\varepsilon\phi_{q,\varepsilon})\right)^{\frac{2}{n}}=\\ &&\frac{ \int_M\left(|\nabla((1-\varepsilon)\phi_{p,\varepsilon}+\varepsilon\phi_{q,\varepsilon})|^2-\frac{h}{r_p^2}((1-\varepsilon)
\phi_{p,\varepsilon}+\varepsilon\phi_{q,\varepsilon})
^2\right)dv_g}{(\int_M|(1-\varepsilon)\phi_{p,\varepsilon}+\varepsilon\phi_{q,\varepsilon}|^{2^*}dv_g)^{\frac{2}{2^*}}}.
\end{eqnarray*}
as
\begin{eqnarray*}
 &\left( nJ((1-\varepsilon)\phi_{p,\varepsilon}+\varepsilon\phi_{q,\varepsilon})\right)^{\frac{2}{n}}&\\
&=\frac{\int_{\R^n}\left(|\nabla U|^2-h(p)\frac{U^2}{|x|^2}\right)dx}{(\int_{\R^n}|U|^{2^*}dx)^{\frac{2}{2^*}}}\left[1+\frac{1}{\int_{\R^n}|U|^{2^*}dx}( [\alpha(n,p)\right.&\\&\left.+\frac{n-2}{n}Scal_g(p)C_3(n,a)I_n^{a(n-2)+1}+
\frac{\varepsilon^2}{(1-\varepsilon)^2}\int_{\R^n}|\nabla U|^2dx\right.&\\&\left.+\alpha(\varepsilon)+o(\varepsilon^2)\right].&
   \end{eqnarray*}
Thus,  by definition of the constant $ D^*$,  we get
\begin{eqnarray*}
&\left( nJ((1-\varepsilon)\phi_{p,\varepsilon}+\varepsilon\phi_{q,\varepsilon})\right)^{\frac{2}{n}}< (nD^*)^{\frac{2}{n}}\left[1+\frac{1}{\int_{\R^n}|U|^{2^*}dx}( [\alpha(n,p)\right.&\\&\left.+\frac{n-2}{n}Scal_g(p)C_3(n,a)I_n^{a(n-2)+1}+\frac{1}{(1-\varepsilon)^2}\int_{\R^n}|\nabla U|^2dx)\varepsilon^2+\alpha(\varepsilon)\right].
\end{eqnarray*}
Hence,
\begin{equation*}
 J(\Phi((1-\varepsilon)\phi_{p,\varepsilon}+\varepsilon\phi_{q,\varepsilon})=D^*+\hbar(\varepsilon),
\end{equation*}
with
\begin{eqnarray*}
&\hbar(\varepsilon)=\frac{nD^*}{2\int_{\R^n}|U|^{2^*}dx}\left( (\alpha(n,p) +\frac{n-2}{n}Scal_g(p)C_3(n,a)I_n^{a(n-2)+1}\right.&\\&\left.+\frac{\varepsilon^2}{(1-\varepsilon)^2}\int_{\R^n}|\nabla U|^2dx+o(\varepsilon^2)\right).&
\end{eqnarray*}
Note that under condition (2) of the lemma, $\hbar (\varepsilon)>0$.\\
 Therefore, we obtain
\begin{equation*}
 D^*<J_h(\Phi((1-\varepsilon)\phi_{p,\varepsilon}+\varepsilon\phi_{q,\varepsilon}))<D^*+g(\varepsilon),
\end{equation*}
 with $g$ is any function such that $g(\varepsilon)>\hbar(\varepsilon)$ and $ g(\varepsilon)\to 0$.
\end{proof}
\subsection{ The map $\beta:\Sigma_\varepsilon\rightarrow M_{R_M}$.} In this subsection, we define a map $\beta:\Sigma_\varepsilon\rightarrow M_{R_M}$. For this aim, we introduce  the barycenter function $ \beta: \NN_h\to \mathbb{R}^n$ defined by
\begin{equation*}
    \beta(u)= \frac{\int_M (x+q-p)|u|^{2^*}dv_g}{\int_M |u|^{2^*}dv_g}.
\end{equation*}
 The function $\beta $ is well defined as $u\neq0$ for all $u \in \NN_h$ and the manifold $M$ is embedded in some $ \mathbb{R}^N$.\\
Now, we prove some properties of the function $\beta$ through  the following lemmas
\begin{Lemma} \label{lem4.3}We have
\begin{equation*}
    \lim_{\varepsilon\to 0}\beta(\I_{\varepsilon}(q))=q
\end{equation*}
\end{Lemma}
\begin{proof} We begin with case where $q=p$. By homogeneity of the function $\beta$, we have
\begin{eqnarray*}
    \beta(\I_{\varepsilon}(p))= \beta(\phi_{p,\varepsilon})=\frac{\int_M x|\phi_{p,\varepsilon}|^{2^*}dv_g}{\int_M |\phi_{p,\varepsilon}|^{2^*}dv_g}.
\end{eqnarray*}
Then
\begin{eqnarray*}
    |\beta(\I_{\varepsilon}(p))- p|&=&|\frac{\int_M x|\phi_{p,\varepsilon}|^{2^*}dv_g}{\int_M |\phi_{p,\varepsilon}|^{2^*}dv_g}-\frac{\int_M p |\phi_{p,\varepsilon}|^{2^*}dv_g}{\int_M |\phi_{p,\varepsilon}|^{2^*}dv_g}|\\
    &\le&\frac{\int_M |x-p||\phi_{p,\varepsilon}|^{2^*}dv_g}{\int_M |\phi_{p,\mu}|^{2^*}dv_g}.
\end{eqnarray*}
For the numerator, we have
\begin{eqnarray*}
\int_M |x-p||\phi_{p,\varepsilon}|^{2^*}dv_g= C(n,a)\int_M     \eta_{p,\delta}r_p(x)\left(\frac{\varepsilon^{a}r_p(x)^{a-1}}{\varepsilon^{2a}+r_p(x)^{2a}}\right)^{\frac{n}{2}-1}dv_g.
\end{eqnarray*}
We repeat the same calculation as in \cite{Maliki}, we get
\begin{eqnarray*}
\int_M |x-p||\phi_{p,\varepsilon}|^{2^*}dv_g&=& \varepsilon \int_{\mathbb{R}^n} |U|^{2^*}dx -Scal_g(p)C_3(n,a)I_n^{a(n-2)+1}\varepsilon^3\\&&+\varepsilon o(\varepsilon^2)+
\alpha(\varepsilon).
\end{eqnarray*}
For the dominator, we have already
\begin{equation*}
\int_M|\phi_{p,\varepsilon}|^{2^*}dv_g= \int_{ \mathbb{R}^n}|U|^{2^*}dx-Scal_g(p)C_3(n,a)I_n^{a(n-2)+1}\varepsilon^2+o(\varepsilon^2)
+\alpha(\varepsilon).
\end{equation*}
 By letting $\varepsilon\to0$, we get the desired equality.\\
 Now, for $q\neq p$, we choose $\delta $ small enough so that $ B(q,2\delta)\cap B(p,2\delta)=\emptyset$. In this situation, the functions $\phi_{p,\varepsilon}$ and $\phi_{q,\varepsilon}$ have disjoint supports.
 Then, similarly as above, we have
 \begin{eqnarray*}
    |\beta(\I_{\varepsilon}(q))-q|&\le&\frac{\int_M |x-p||(1-\varepsilon)\phi_{p,\varepsilon}+\varepsilon\phi_{q,\varepsilon}|^{2^*}dv_g}{\int_M |(1-\varepsilon)\phi_{p,\varepsilon}+\varepsilon\phi_{q,\varepsilon}|^{2^*}dv_g}.
\end{eqnarray*}
 Since the functions  $\phi_{p,\varepsilon}$ and $\phi_{q,\varepsilon}$ have disjoint supports, we have
 \begin{eqnarray*}
&&\int_M |x-q||\phi_{q,\varepsilon}|^{2^*}dv_g\\&=& (1-\varepsilon)^{2^{*}}\int_M |x-p||\phi_{p,\varepsilon}^{2^*}dv_g+\varepsilon^{2^*}\int_M |x-p||\phi_{q,\varepsilon}|^{2^*}dv_g\\
&\le &  (1-\varepsilon)^{2^{*}}\int_M |x-p||\phi_{p,\varepsilon}^{2^*}dv_g+\varepsilon^{2^*}\int_M |x-p||\phi_{q,\varepsilon}|^{2^*}dv_g\\&&+ \varepsilon^{2^*}\int_M |x-q||\phi_{q,\varepsilon}|^{2^*}dv_g+
\varepsilon^{2^*}|q-p|\int_M |\phi_{q,\varepsilon}|^{2^*}dv_g
\end{eqnarray*}
Like before, we have
\begin{equation*}
\int_M|\phi_{q,\varepsilon}|^{2^*}dv_g= \int_{ \mathbb{R}^n}|U|^{2^*}dx-Scal_g(q)C_3(n,a)I_n^{a(n-2)+1}\varepsilon^2+o(\varepsilon^2),
+\alpha(\varepsilon).
\end{equation*}
and
\begin{eqnarray*}
\int_M |x-q||\phi_{q,\varepsilon}|^{2^*}dv_g&=& \varepsilon \int_{\mathbb{R}^n} |U|^{2^*}dx -Scal_g(q)C_3(n,a)I_n^{a(n-2)+1}\varepsilon^3\\&&+\varepsilon o(\varepsilon^2)+
\alpha(\varepsilon).
\end{eqnarray*}
Hence, by using the expansion
\begin{eqnarray*}
&\int_M|(1-\varepsilon)\phi_{p,\varepsilon}+\varepsilon\phi_{q,\varepsilon}|^{2^*}dv_g &\\ &=((1-\varepsilon)^{2^*}+\varepsilon^{2^*})\int_{\mathbb{R}^n}|U|^{2^*}dx-(1-\varepsilon)^{2^*}Scal_g(p)C_1(n,a)
I_n^{a(n-2)+1}\varepsilon^2&\\ \nonumber &+o(\varepsilon^2)+\alpha(\varepsilon)),&
\end{eqnarray*}
we get
 \begin{eqnarray*}
   |\beta(\I_{\varepsilon}(q))-q|\to 0 \text{ as } \varepsilon\to 0.
\end{eqnarray*}
\end{proof}

\begin{Lemma} \label{lem4.4}For any $\eta\in(0,1)$ and for every $u\in \Sigma_\varepsilon$, there  exists a point $q\in M$ such that
\begin{equation*}
    \int_{B(q,\frac{r_M}{2})}|u|^{2^*}dv_g>(1-\eta)(S_{h(p)})^{\frac{n}{2}}
\end{equation*}
\end{Lemma}
\begin{proof} Suppose by contradiction that there exist $\eta\in(0,1)$ , a sequence $\varepsilon_m\to 0$ as $m\to\infty$ and a sequence $u_m\in \Sigma_{\varepsilon_m}$ such that for all $q\in M$
\begin{equation*}
  \int_{B(q,\frac{r_M}{2})}|u_m|^{2^*}dv_g\le(1-\eta)(S_{h(p)})^{\frac{n}{2}}.
\end{equation*}
By the Ekland variational principle, we can assume that $D_{\NN_h}J_h(u_m)\to 0$ as $m\to \infty$. Since $D^*<J_h(u_m)<D^*+g(\varepsilon_m)$, for some $g(\varepsilon_m)>0$ and $g(\varepsilon_m)\to0$ as $m\to0$ and since the manifold $\NN_h$  defines a natural constraint for the functional $J_h$ ( see \cite{Amrosetti}), we can assume that $u_m$ is a P-S sequence of $J_h$ at level $D^*$. Thus by corollary \ref{cor2}, there  exists a sequence of reals $R_m \to 0$ as $m\to\infty$ and a sequence $w_m\in H^2_1(M)$ that converges strongly to $0$ in $\in H^2_1(M)$ such that
 \begin{equation*}
    u_m= \phi_{p,R_m\varrho}+w_m.
 \end{equation*}
 Hence, by applying the inequality
 \begin{equation*}
 (a+b)^{2^*}\le a^{2^*}+b^{2^*}+2^*a^{2^*-1}b+2^*ab^{2^*-1}, a\ge0,b\ge 0,
 \end{equation*}
 and  by using the fact that $w_m\to0$ strongly in $H^2_1(M)$, we obtain
 \begin{equation*}
  \int_{B(q,\frac{r_M}{2})}|\phi_{p,R_m\varrho}|^{2^*}dv_g\le(1-\eta)(S_{h(p)})^{\frac{n}{2}}.
\end{equation*}
Put $\varepsilon_m^*=R_m\varrho $. Then, $\varepsilon_m^*\to 0$ as $m\to \infty$. Thus, by using the expansion \eqref{4.20}, we have
\begin{eqnarray*}
\int_M|\phi_{p,\varepsilon_m^*}|^{2^*}dv_g&=& \int_{\R^n}|U|^{2^*}dx-Scal_g(p)C_3(n,a)I_n^{a(n-2)+1}(\varepsilon_m^*)^2
\\&&+o((\varepsilon_m^*)^2)+\alpha(\varepsilon_m^*),
\end{eqnarray*}
 As the function $U$ is a positive solution of \eqref{eqs}, we get
 \begin{eqnarray*}
\int_M|\phi_{\varepsilon_m^*,p}|^{2^*}dv_g&=& (nD^*)^{\frac{n}{2}}-Scal_g(p)C_3(n,a)I_n^{a(n-2)+1}(\varepsilon_m^*)^2
\\&&+o((\varepsilon_m^*)^2)+\alpha(\varepsilon_m^*)
\end{eqnarray*}
 Recall that the function $U$ is supported in $B(p,\delta)$, then by choosing $\delta$ small, we obtain by
  \begin{equation*}
  (S_{h(p)})^{\frac{n}{2}}-Scal_g(p)C_3(n,a)I_n^{a(n-2)+1}(\varepsilon_m^*)^2 +o((\varepsilon_m^*)^2)+\alpha(\varepsilon_m^*)\le(1-\eta)(S_{h(p)})^{\frac{n}{2}}.
\end{equation*}
 Hence, by letting $m\to\infty$, we get the contradiction:
\begin{equation*}
(S_{h(p)})^{\frac{n}{2}}\le(1-\eta)(S_{h(p)})^{\frac{n}{2}}.
\end{equation*}
\end{proof}
\begin{Lemma}\label{lem4.5} For $\epsilon$  small, $\beta(u)\in M_r$ for every function $u\in\Sigma_\varepsilon$.
\end{Lemma}
\begin{proof} It suffices to prove that for every $u\in\Sigma_\varepsilon$,
\begin{equation*}
    |\beta(u)-p|\le r_M.
\end{equation*}
Let $u\in\Sigma_\varepsilon$, by lemma \ref{lem4.4}, we get that for any $\eta\in (0,1)$
\begin{equation*}
\frac{\int_{B(p,\frac{r_M}{2})}|u|^{2^*}dv_g}{\int_{M}|u|^{2^*}dv_g}>
\frac{(1-\eta)(S_{h(p)})^{\frac{n}{2}}}{n(D^*+g(\varepsilon))}
=\frac{(1-\eta)(S_{h(p)})^{\frac{n}{2}}}{(S_{h(p)})^{\frac{n}{2}}+ng(\varepsilon)}
\end{equation*}
Then, we obtain
\begin{eqnarray*}
    |\beta(u)-p|&=&| \frac{\int_M(x-p)|u|^{2^*}dv_g}{\int_M|u|^{2^*}dv_g}|\\
    &\le&\frac{\int_{B(p,\frac{r_M}{2})}|x-p||u|^{2^*}dv_g}{\int_M|u|^{2^*}dv_g}+ \frac{\int_{M\setminus B(p,\frac{r_M}{2})}|x-p||u|^{2^*}dv_g}{\int_M|u|^{2^*}dv_g} \frac{r_M}{2}\\
    &\le&\frac{r_M}{2}+D(M)(1- \frac{\int_{B(p,\frac{r_M}{2})}|u|^{2^*}dv_g}{\int_{M}|u|^{2^*}dv_g})   \\
    &\le&  \frac{r_M}{2}+D(M)(1- \frac{(1-\eta)(S_{h(p)})^{\frac{n}{2}}}{(S_{h(p)})^{\frac{n}{2}}+ng(\varepsilon)})
\end{eqnarray*}
 where $D(M)$ is the diameter of $M$. Thus, in order to get the conclusion, it suffices to choose $\eta$ and $\varepsilon $ small enough so that
\begin{equation*}
D(M)(1- \frac{(1-\eta)(S_{h(p)})^{\frac{n}{2}}}{(S_{h(p)})^{\frac{n}{2}}+ng(\varepsilon)})\le \frac{r_M}{2}.
\end{equation*}
 \end{proof}
\section{Proof of the main result}
\begin{proof} By Lemmas \ref{lem-inj} and \ref{lem4.5} the maps $\I_\varepsilon: M\rightarrow \Sigma_\varepsilon$ and $\beta: \sigma_\varepsilon\rightarrow M_{r_M}$ are well defined. Moreover, by lemma \ref{lem4.3} the composition $\beta \circ\I_\varepsilon: M\rightarrow M_{r_M}$ is well defined  and is homotopic to the identity. Thus, by the properties of  Lusternik-Schnirelmann category, $Cat \Sigma_\varepsilon\ge Cat(M)$.
Since the Palais-Smale conditions are satisfied in the set $\Sigma_\varepsilon$, by theorem \ref{thm-L-S} there are at least  $ cat(M)$ critical  points of the functional $J_h$.\\
It remains, to achieve the proof of the theorem, to prove that there exists another critical point $u$ with $J_h(u)>D^*+g(\varepsilon)$. For this task, following \cite{Benci}, we construct a set $P_\varepsilon$ which is contractible in $\NN_h\cap J_h^c$.\\
 Let $V\in D^{1,2})(\mathbb{R}^n)$ be any function and define on the manifold $M$ the function
 \begin{equation*}
    V_\varepsilon(x)= \eta_{p,\delta}(\exp_p^{-1}(x))V(\varepsilon^{-1}\exp_p^{-1}(x), x\in B(p,\delta).
   \end{equation*}
Put $ \varphi_\varepsilon=(1-\varepsilon)\phi_{p,\varepsilon}+\varepsilon\phi_{q,\varepsilon} $ and define the set
\begin{equation*}
    \Omega_\varepsilon=\{ (1-t)\varphi_\varepsilon+tV_\varepsilon , t\in [0,1]\}
\end{equation*}
Consider $P_\varepsilon$, the projection of $\Omega_\varepsilon$ on the Nehari manifold $\NN_h$
\begin{equation*}
 P_\varepsilon=\{ \Phi(\omega_\varepsilon), \omega_\varepsilon\in \Omega_\varepsilon \}
\end{equation*}
We notice immediately that $ \I_\varepsilon(M)\subset P_\varepsilon$, $P_\varepsilon$ is compact and contractible in $\NN_h$. Then, put
\begin{equation*}
    c_\varepsilon=\sup_{u_\varepsilon\in P_\varepsilon}J_h(u).
\end{equation*}
We need to prove that $c_\varepsilon$ is bounded with respect to $\varepsilon$. For this aim, for $u\in\Omega_\varepsilon$ write
\begin{equation*}
    J_h(\Phi(u))= \frac{1}{n}\left(\frac{\int_M(|\nabla u|^2-\frac{h}{\rho_p^2}u^2)dv_g}{(\int_M|u|^{2^*}dv_g)^\frac{2}{2^*}}\right)^{\frac{n}{2}}.
\end{equation*}
We have
\begin{eqnarray}\label{5.24}
&&\nonumber \int_M|\nabla u_\varepsilon|^2dv_g\\ \nonumber &=&t^2\int_M|\nabla V_\varepsilon|^2dv_g+(1-t)^2  \int_M|\nabla \varphi_\varepsilon|^2dv_g+2t(t-1)\int_M\nabla V_\varepsilon.\nabla\varphi_\varepsilon dv_g\\&&
\\ \nonumber&\le&\int_{\mathbb{R}^n}|\nabla V|^2dx+ 2\int_{\mathbb{R}^n}|\nabla U|^2dx+2(\int_{\mathbb{R}^n}|\nabla V|^2.\int_{\mathbb{R}^n}|\nabla U|^2dx)^{\frac{1}{2}}+K_1.
\end{eqnarray}
Also, we have
\begin{eqnarray}\label{5.25}
&&\nonumber|\int_M\frac{h}{\rho_p^2}u^2_\varepsilon dv_g|\\ \nonumber&\le& t^2\int_M \frac{h}{\rho_p^2}V^2_\varepsilon dv_g+(1-t)^2\int_M \frac{h}{\rho_p^2}\varphi^2_\varepsilon dv_g+2t(1-t)\sup_M|h|K(n,-2,\delta)\\ \nonumber&&((\int_M |V_\varepsilon|^2dv_g\int_M|\nabla \varphi_\varepsilon|^2 dv_g)^\frac{1}{2}+(\int_M |\varphi_\varepsilon|^2dv_g\int_M\nabla|V_\varepsilon|^2 dv_g)^{\frac{1}{2 }})\\&& \\ \nonumber
&\le& h(p)( \int_{\mathbb{R}^n}|V|^2dx+ \int_{\mathbb{R}^n}|\nabla U|^2dx)+C((\int_{\mathbb{R}^n}|V|^2dx\int_{\mathbb{R}^n}|\nabla U|^2dx)^{\frac{1}{2}}\\ \nonumber&+&
(\int_{\mathbb{R}^n}|U|^2dx\int_{\mathbb{R}^n}|\nabla V|^2dx)^{\frac{1}{2}}+K_2.
\end{eqnarray}
Moreover, there exists $\varepsilon_o$ and $\varepsilon_1$ such that
\begin{equation*}
\int_M|\varphi_\varepsilon|^{2^*}dv_g \ge \int_{\mathbb{R}^n}|U|^{2^*}dx- Scal_g(p)C_3(n,a)I_n^{a(n-2)+1}\varepsilon_o^2>0,
      \end{equation*}
and
\begin{eqnarray*}
  \int_M|V_\varepsilon|^{2^*}dv_g \ge\int_{\mathbb{R}^n}|V|^{2^*}dx- Scal_g(p)C_3(n,a)I_n^{a(n-2)+1}\varepsilon_1^2>0.
\end{eqnarray*}
 Then, since $V_\varepsilon$ and $\varphi_\varepsilon$ ar positive, we get
\begin{eqnarray*}
&&\int_M|u_\varepsilon|^{2^*}dv_g\\&\ge& \max( t^{2^*}\int_M|V_\varepsilon|^{2^*}dv_g,(1-t)^{2^*}\int_M|\varphi_\varepsilon|^{2^*}dv_g)\\
&\ge &\frac{1}{2^*}\min(\int_{\mathbb{R}^n}|U|^{2^*}dx-K_3,\int_{\mathbb{R}^n}|V|^{2^*}dx-K_4),
\end{eqnarray*}
where
\begin{equation*}
K_3=Scal_g(p)C_3(n,a)I_n^{a(n-2)+1}\varepsilon_o^2 \text{ and } k_4=Scal_g(p)C_3(n,a)I_n^{a(n-2)+1}\varepsilon_1^2.
\end{equation*}
which gives together with estimates \eqref{5.24} and \eqref{5.25} the thesis.
\end{proof}

\end{document}